\documentclass{amsart}

\usepackage{amssymb}

\newtheorem{theorem}{Theorem}[section]
\newtheorem{lemma}[theorem]{Lemma}
\newtheorem{cor}[theorem]{Corollary}
\theoremstyle{definition}
\newtheorem{definition}[theorem]{Definition}
\newtheorem{example}[theorem]{Example}
\newtheorem{question}[theorem]{Question}

\numberwithin{equation}{section}

\usepackage{enumitem}
\DeclareMathOperator{\cl}{cl}
\DeclareMathOperator{\CL}{CL}
\begin{document}

\title[PSEUDOCOMPACTNESS OF HYPERSPACES OF SUBSETS OF $\beta \omega$]{SMALL CARDINALS AND THE PSEUDOCOMPACTNESS OF HYPERSPACES OF SUBSPACES OF \MakeLowercase{$\beta \omega$}}


\author{ORTIZ-CASTILLO, Y. F.}
\address{Instituto de Matem\'atica e Estat\'istica, Universidade de S\~ao Paulo \\ Rua do Mat\~ao, 1010, CEP 05508-090, S\~ao Paulo, Brazil}
\curraddr{}
\email{jazzerfoc@gmail.com}
\thanks{}

\author{RODRIGUES, V. O.}
\address{Instituto de Matem\'atica e Estat\'istica, Universidade de S\~ao Paulo \\ Rua do Mat\~ao, 1010, CEP 05508-090, S\~ao Paulo, Brazil}
\curraddr{}
\email{vinior@ime.usp.br}
\thanks{}
 
\author{TOMITA, A. H.}
\address{Instituto de Matem\'atica e Estat\'istica, Universidade de S\~ao Paulo \\ Rua do Mat\~ao, 1010, CEP 05508-090, S\~ao Paulo, Brazil}
\curraddr{}
\email{tomita@ime.usp.br}
\thanks{}
\date{}

\dedicatory{}

\commby{}

\begin{abstract}
We study the relations between a generalization of pseudocompactness, named $(\kappa, M)$-pseudocompactness, the countably compactness of subspaces of $\beta \omega$ and the pseudocompactness of their hyperspaces.  We show, by assuming the existence of $\mathfrak c$-many selective ultrafilters, that there exists a subspace of $\beta \omega$ that is $(\kappa, \omega^*)$-pseudocompact for all $\kappa<\mathfrak c$, but $\CL(X)$ isn't pseudocompact. We prove in ZFC that if $\omega\subseteq X\subseteq \beta\omega$ is such that $X$ is $(\mathfrak c, \omega^*)$-pseudocompact, then $\CL(X)$ is pseudocompact, and we further explore this relation by replacing  $\mathfrak c$ for some small cardinals. We provide an example of a subspace of $\beta \omega$ for which all powers below $\mathfrak h$ are countably compact whose hyperspace is not pseudocompact, we show that if $\omega \subseteq X$, the pseudocompactness of $\CL(X)$ implies that $X$ is $(\kappa, \omega^*)$-pseudocompact for all $\kappa<\mathfrak h$, and provide an example of such an $X$ that is not $(\mathfrak b, \omega^*)$-pseudocompact.
\end{abstract}

\maketitle

\section{Introduction}

The letters $\kappa$, $\lambda$, $\mu$ and $\theta$ represent cardinal numbers, the letters $\alpha$, $\beta$, $\gamma$ and similar represent ordinal numbers, and $\operatorname{cf} (\alpha)$ denote the cofinality of the ordinal $\alpha$. With $\omega$ we denote the first infinite cardinal, $\omega_{1}$ is the first non-countable cardinal and $\mathfrak c$ represent the continuum.  Given a set $X$ and a cardinal number $\kappa$, $[X]^{\kappa}$ represent the family $\{A\subseteq X: |A|=\kappa\}$. We define the cardinals $\mathfrak b, \mathfrak p, \mathfrak t, \mathfrak h$ as in \cite{CombinatorialBlass}. Let $A, B\subseteq \omega$, we say that $A\subseteq^*B$ if $A\setminus B$ is finite. If $\mathcal A\subseteq [\omega]^\omega$, a pseudo-intersection of $\mathcal A$ is a set $B \in [\omega]^{\omega}$ such that for all $A \in \mathcal A$, $B\subseteq^*A$. We say that $\mathcal A\subseteq [\omega]^{\omega}$ has the Strong Finite Intersection Property (SFIP) if the intersection of any finite nonempty subset of $\mathcal A$ is infinite. The pseudo-intersection number $\mathfrak p$ is the smallest cardinality of a family with SFIP but with no pseudo-intersection. Given a cardinal $\kappa$, a $\subseteq^*$ $\kappa$-tower is a family $\{A_\alpha: \alpha<\kappa\}$ of infinite subsets of $\omega$ such that whenever $\alpha<\beta<\kappa$, $A_\alpha\subseteq^*A_\beta$. The tower number $\mathfrak t$ is the smallest cardinal of a $\subseteq^*$ $\kappa$-tower without a pseudo-intersection. Both $\mathfrak p, \mathfrak t$ are regular, $\omega_1\leq p\leq t\leq \mathfrak c$ and M. Malliaris and S. Shelah have recently proved that they are, in fact, the same \cite{PisT}. We say that $\mathcal A\subseteq[\omega]^\omega$ is open dense if for every $X \in [\omega]^\omega$ there exists $A \in \mathcal A$ such that $X\subseteq A$ and if for every $A, B \in [\omega]^\omega$ if $A\in \mathcal A$ and $B \subseteq^*A$ then $B \in \mathcal A$. So, $\mathfrak h$ is the smallest cardinal for which there is a collection of $\mathfrak h$ open dense sets whose intersection is empty. Finally, given $f, g \in\,^\omega \omega$, we say that $f<^*g$ if there exists $m \in \omega$ such that for every $n \geq m$, $f(n)<g(m)$. We say that $\mathcal B \subseteq\, ^\omega \omega$ is bounded if there exists $g\in ^\omega \omega$ such that for every $f \in \mathcal B$, $f<^*g$. If $\mathcal B$ is not bounded, we say that $\mathcal B$ is unbounded and then $\mathfrak b$ is the smallest cardinality of an unbounded family.  It is true that $\omega_1\leq \mathfrak t\leq \mathfrak h\leq \mathfrak b\leq \mathfrak c$ and it is consistent that each inequality is strict (see \cite{CombinatorialBlass}). A family $\mathcal A\subseteq [\omega]^\omega$ is almost disjoint if $\mathcal A$ is infinite and for every two distinct $A, B\in \mathcal A$, $A\cap B$ is finite. A mad family is a maximal almost disjoint family in the $\subseteq$ sense. Given an almost disjoint family $\mathcal A$, the Mr\'{o}wka-Isbell space of $\mathcal A$ is $\Psi(\mathcal A)=\mathcal A\cup \omega$, where $\omega$ is open and discrete, and for every $A \in \mathcal A$, $\{\{A\}\cup (A\setminus n): n \in \omega\}$ is a local basis of the point $A$. Recall that $\omega$ is dense in $\Psi(\mathcal A)$,  $\Psi(\mathcal A)$ is Tychonoff, locally compact and zero-dimensional . The  \emph{Stone-\v{C}ech} compactification $\beta \omega$  of the discrete countable space $\omega$  will be identified with the set of all ultrafilters on $\omega$ and its remainder $ \omega^*$ will be identified with the set of all free ultrafilters on $\omega$.

\smallskip

Given a topological space $X$, the (Vietoris) Hyperspace of $X$ consists of the set $\CL(X)$ of all nonempty closed subsets of $X$ equipped with the topology generated by the sets of the form $A^+=\{F \in \CL(X): F\subseteq A\}$ and $A^-=\{F \in \CL(X): F\cap A\neq \emptyset\}$, where $A\subseteq X$ is open. A basic  open set of the Vietoris topology is of the form
$
\left\langle U_1,\ldots,U_n\right\rangle=  \big(\bigcup_{i \leq n}U_{k}\big)^{+} \cap \big(\bigcap_{i \leq n}U_{i-}\big),
$
where $U_{1}$,\dots,$U_{n}$ are nonempty open subsets of $X$. $\CL(X)$ is $T_1$ if, and only if $X$ is Hausdorff, $\CL(X)$ is Hausdorff if, and only if $X$ is regular, and $X$ is compact if, and only if $\CL(X)$ is compact (see \cite{Michael1951}), so it is natural to ask whether there are similar results regarding weaker compactness-like properties.

We say that $X$ is pseudocompact if  every continuous function from $X$ into $\mathbb R$ is bounded\footnote{For Tychonoff spaces, this statement is equivalent to the claim that there is no infinite discrete family of open sets. Many hyperspaces are not Tychonoff, however, this equivalence still holds for hyperspaces (\cite{Ginsburg75})}, also recall that $\Psi(\mathcal A)^\omega$ is pseudocompact if, and only if $\mathcal A$ is a mad family. Following the paper \cite{Ber}, given a space $X$, $p\in\mathbb{N}^{*}$ and a sequence $(S_{n})_{n\in \mathbb{N}} $ of nonempty subsets of $X$, we say that  $x\in X$ is a  $p$-limit  of $(S_{n})_{n\in \mathbb{N}}$, in symbols $x = p-\lim \ S_n$, if  $\{n\in \mathbb{N}: S_{n}\cap W\neq\emptyset\}\in p$ for each neighborhood $W$ of $x$. In particular $p$-limits of sequences of points on Hausdorff spaces (defined as $p$-limits of singular sets), are unique when exist.
Those notions used and not defined in this article have the meaning given to them in \cite{En} and \cite{kunen1980set}.

\medskip

In  \cite{Ginsburg75}, J. Ginsburg has raised the question whether there is a relationship between the pseudocompactness of $X^\omega$ and  $\CL(X)$. He also proved that if every power of a space $X$ is countably compact then so is $\CL(X)$, and that if $\CL(X)$ is countably compact (pseudocompact) then so is every finite power of $X$.  Michael Hru\v{s}\'{a}k, Fernando Hern\'{a}ndez-Hern\'{a}ndez and Iv\'{a}n Mart\'{i}nez-Ruiz proved in \cite{PseudoHyper} that under $\mathfrak h<\mathfrak c$, there exists a mad family $\mathcal A$ such that $\CL(\Psi(\mathcal A))$ is not pseudocompact, and that under $\mathfrak p=\mathfrak c$, for every mad family $\mathcal A$, $\CL(\Psi(\mathcal A))$ is pseudocompact. They also constructed, in ZFC, a subspace of $\beta \omega$ whose hyperspace is not compact but its $\omega$-th power is pseudocompact, showing that under the axioms of ZFC, it is not true that the pseudocompactness of $X^\omega$ implies the pseudocompactness of $\CL(X)$. Also J. Cao, T. Nogura, A. H. Tomita proved in \cite{tomita2004}  that if $X$ is a homogeneous Tychonoff space such that $\CL(X)$ is pseudocompact then $X^\omega$ is pseudocompact.

So, it is natural to ask whether the countable compactness of the former implies the pseudocompactness of the latter. In this article, we explore this question.

Following \cite{NYAS:NYAS22}, if $(B_n:\, n\in \omega)$ is a sequence of subsets of a space $X$ and $p \in \omega^*$, then we will say that the point $x \in X$ is a $p$-limit  of $(B_n:\, n\in \omega)$  if for every neighborhood $V$ of $x$, $\{n \in \omega: V\cap B_n\neq \emptyset\} \in p$.  If $\kappa$ is a cardinal and $M\subseteq \omega^*$, we say that a space $X$ is $(\kappa, M)-$pseudocompact if and only if for every family $\{(V^\alpha_n: n \in \omega):\alpha<\kappa\}$ of sequences of open subsets of $X$ there exists $p \in M$ such that for every $\alpha<\kappa$ there is a $p$-limit point of $V^\alpha_n$ in $x$. We say that $Y\subseteq X$ is relatively countably compact in $X$ if every $A \in [Y]^\omega$ there is an accumulation point of $Y$ in $X$. Notice that if $D\subseteq X$ is an open discrete dense subset of $X$ and $\kappa$ is a cardinal, then $D^\kappa$ is relatively countably compact in $X$ if and only if $X$ is $(\kappa, \omega^*)-$pseudocompact. the same article proves that there $X$ is $(\omega, \omega^*)$-pseudocompact if, and only if $X^\omega$ is pseudocompact. It is well known that if $X^\omega$ is pseudocompact then every power of $X$ is pseudocompact, however, it is not true that $X$ is $(\omega, \omega^*)$-pseudocompact then every power of $X$ is pseudocompact.

We will construct, in ZFC, a subspace $X$ of $\beta \omega$ such that $X^\kappa$ is countably compact for every $\kappa<\mathfrak h$, but $\CL(X)$ is not pseudocompact, explore the question whether this is possible for $\kappa=\mathfrak t$ and build a machinery for trying to answer such questions. We show that if $\omega \subseteq X$, the pseudocompactness of $\CL(X)$ implies that both $X$ and $\CL(X)$ are $(\kappa, \omega^*)$-pseudocompact for all $\kappa<\mathfrak h$, and we provide an example of such an $X$ that is not $(\mathfrak b, \omega^*)$-pseudocompact.

We also develop a different proof of the fact that under $\mathfrak p=\mathfrak c$, every Hyperspace of Mr\'{o}wka-Isbell space of a mad family is pseudocompact.

\section{Countably compactness of products of $X$ not implying that $\CL(X)$ is pseudocompact}

We begin this section proving a  lemma that will be required for the main theorems.

\begin{lemma}\label{ContZp}
Suppose $\omega\subseteq X\subseteq \beta \omega$ and let $p$ be a free ultrafilter over $\omega$. Let $(C_n: n \in \omega)$ be a sequence of finite pairwise disjoint  nonempty subsets of $\omega$, let $\mathcal G=\{g \in \omega^\omega: g(n) \in C_n\}$ and let $Z_p=\{p-\lim g:  g \in \mathcal G\}$. Then the following statements are equivalent:
\begin{enumerate}
 \item $(C_n:\, n\in \omega)$  has  $p$-limit in $\operatorname{CL}(X)$,
 
 \item $\operatorname{cl_X} Z_p=p-\lim \ C_n$, and
 
 \item $Z_p\subseteq X$.
 \end{enumerate}
In addition, if  $|C_n|\leq|C_{n+1}|$ for every $n\in \omega$ and  $|C_n|\rightarrow \infty$, then $|Z_p|=\mathfrak c$.
\end{lemma}
\begin{proof}
First we will show that $Z_p$ is discrete in $\beta \omega$. Let $z \in Z_p$ and $g \in \mathcal G$ such that $z=p-\lim g$. Then $z$ belongs to the clopen set $W=\cl_{\beta \omega}\{g(n): n \in \omega\}\subseteq \beta \omega$. We claim that $Z_p \cap W=\{z\}$. Pick $h \in \mathcal G$ and observe that, since the sets $\{n \in \omega: g(n)=h(n)\}$ and $ \{n \in \omega: g(n)\neq h(n)\}$ are a partition of $\omega$, then one and only one of those sets belongs to $p$. If $\{n \in \omega: g(n)=h(n)\} \in p$, then $p-\lim h=p-\lim g=z$. Otherwise let $A=\{n \in \omega: g(n)\neq h(n)\} \in p$. Since the sets $C_n$ are pairwise disjoint, it follows that $\{g(n): n \in A\}\cap \{h(n): n \in A\}=\emptyset$, which implies that $\{g(n): n \in \omega\} \cap \{h(n): n \in A\}=\emptyset$. As $\{h(n): n \in A\}\in p-\lim \ h$ we obtain that $p-\lim h\notin W$. Now suppose that the $p$-limit of  the sequence $(C_n:\, n\in \omega)$ exists in $\operatorname{CL}(X)$ and let $F= p-\lim \ C_n$.

\begin{flushleft}
\textbf{Claim 1:} $F\subseteq \cl_{\beta\omega} Z_p$.
\end{flushleft}
{\bf Proof of Claim:}  Given $x\in F$ and $U$ open set of $\beta \omega$ which contains $x$, let $V$ a open set such that $x\in V\subseteq \cl_{\beta \omega}(V)\subseteq U$. Thus  $F\in \langle V, \beta \omega\setminus\{x\} \rangle$. Since $F= p-\lim \ C_n$, then  $B=\{n\in  \omega: C_n\cap V\neq \emptyset\}\in p$. Let $g:\omega \to \omega$ such that $g(n)\in C_n\cap V$ when $n\in B$ and $g(n)\in C_n$ otherwise. Of course $g\in \mathcal G$ and $p-\lim \ g \in \cl_{\beta \omega}(V)$. So $p-\lim \ g \in U \cap Z_p$, which implies that $z\in \cl_{\beta\omega} Z_p$. Then  $F\subseteq \cl_{\beta\omega} Z_p$. 

\begin{flushleft}
\textbf{Claim 2:} $Z_p\subseteq F$.
\end{flushleft}
{\bf Proof of Claim:}  Suppose by contradiction that $z \in Z_p \setminus F$ and let $g \in \mathcal G$ be such that $z=p-\lim g$. Since $Z_p$ is discrete in $\beta \omega$, we have that $z \notin \cl_{\beta \omega}(Z_p\setminus \{z\})$. So there are two disjoint clopen sets $V$ and $W$ such that $\cl(Z_p\setminus \{z\})\subseteq V$ and $z\in W$. Thus 
$$
\{n \in \omega: C_n\in V^+ \}=\{n \in \omega: C_n\subseteq V \}\subseteq \{n \in \omega: g(n)\notin  W \}\notin p
$$
but this is not possible  because $F\in V^+$ by Claim 1.

\medskip

Now  note that $(1) \Rightarrow (2)$ follows from Claims 1 and 2; and $(2) \Rightarrow (3)$ is evident.
To show  $(3) \Rightarrow (1)$ take a basic neighborhood  $\langle U_0, \ldots ,U_{k}\rangle$ of $\cl_X Z_p$ where $U_i$ is a clopen set of $X$ for each $i \leq k$. Pick $z_i \in U_i \cap Z_p$  and fix $g_i \in \mathcal G$ that witnesses that  $z_i \in Z_p$ for each $i \leq k$.  It follows that 
$$
\{ n \in \omega :\, C_n \cap U_i \neq \emptyset \} \supseteq \{n\in \omega:\, g_i(n) \in U_i \}\in p
$$ 
for each $i \leq k$. Also observe  that, if $A=\{ n \in \omega:\, C_n \not\subseteq  U_0\cup \ldots \cup U_{k} \}\in p$, then there is  $ g\in \mathcal G$ such that $g(n) \in C_n \setminus (U_0 \cup \ldots \cup U_{k})$ for each $n\in A$.  So the $p$-limit of $g$ does not belong to $(U_0 \cup \ldots \cup U_{k} )$ which implies that $\cl_X Z_p \setminus (U_0 \cup \ldots \cup U_{k} )\neq \emptyset$, a contradicion. So we obtain that $\{ n \in \omega:\, C_n \in \langle U_0, \ldots U_{k-1}\rangle \}\in p$ because $p$ is an ultrafilter. Thus $\cl_X Z_p$ is the $p$-limit of $(C_n:\, n \in \omega)$.

\medskip

Finally assume that $|C_n|\leq|C_{n+1}|$ for every $n\in \omega$ and  $|C_n|\rightarrow \infty$. Then there is a strictly  increasing sequence of natural numbers $(k_m: m\in \omega)$ such that  $k_0=0$ and  $|C_{k_m}|\geq 2^m$ for every $m\in \omega$. Let $s(n)$ be the unique natural number $m$ such that  $k_m\leq n<k_{m+1}$ for each $n \in \omega$ ($s(n)$ is well defined because $(k_m: m\in \omega)$ strictly  increasing).  Observe  that $|C_{n}|\geq2^{s(n)}$ and $s(k_n)=n$ for every $n \in \omega$. Let $\{e_n(0), \dots, e_n(2^{s(n)}-1)\}$ be a subset of $C_n$ of cardinality $2^{s(n)}$ for every $n \in \omega$. Also, for a given $\sigma \in 2^\omega$, let $f_\sigma:\omega \to \omega$ be the fundtion defined by $f_\sigma(n)=\sum_{i<s(n)}2^i\sigma(i)$. Hence $\{f_\sigma[\omega]: \sigma \in 2^\omega\}$ is an AD family of cardinality $\mathfrak c$ of $\omega$. Finally, let  define the function $g_\sigma:\omega \to \omega$ by $g_\sigma(n)=e_n(f_\sigma(n))$ for every $n \in \omega$ and every $\sigma \in 2^\omega$. Observe that $g_\sigma\in\mathcal G$ for every $\sigma \in 2^\omega$, furthermore, since the $C_{n}$ are disjoint, we have that $g_\sigma(n)=g_{\sigma'}(m)$ iff $n=m$ and $f_\sigma(n)=f_{\sigma'}(n)$, so, the family $\{g_\sigma[\omega]: \sigma \in 2^\omega\}$ is also an AD family of cardinality $\mathfrak c$ of $\omega$. Thus $\{\cl_{\beta \omega}(g_\sigma[\omega])\setminus \omega: \sigma \in 2^\omega\}$ is a pairwise disjoint family of closed sets. Clearly $z_\sigma=p-\lim \ g_\sigma \in (\cl_{\beta \omega} g_\sigma[\omega])\setminus \omega$ for every $\sigma \in 2^\omega$ so, $|\{z_\sigma: \sigma \in 2^\omega\}|=\mathfrak c$. Since $\{z_\sigma: \sigma \in 2^\omega\}\subseteq Z_p$ we obtain that $|Z_p|=\mathfrak c$.
\end{proof}


Now we will establish conditions for uncountable cardinal numbers $\mu \leq \mathfrak c$ under which there exist subspaces of $\beta \omega$  with  countably compact $< \mu$-powers but such that  its Vietoris hyperspace is not even pseudocompact.

\begin{theorem}\label{smallbig} 
Let  $\mu$  and $\lambda$ be two uncountable cardinals such that:

\begin{enumerate}[label=\alph*)]
	\item $\omega_1\leq \mu\leq  \mathfrak c\leq\lambda\leq 2^{\mathfrak c}$, $\lambda^{<\mu}\leq \lambda$, $\operatorname{cf}(\lambda)\geq \mu$, and
    \item for every infinite cardinal $\kappa<\mu$ and every $Y\subseteq [\omega^*]^{<\lambda}$,  all the sequences in $(\beta \omega\setminus Y)^\kappa$ have an accumulation point in $(\beta \omega\setminus Y)^\kappa$.
\end{enumerate}
Then there exists  $X\subseteq \beta \omega$ such that $\omega\subseteq X$, $|X|=\lambda$ and $X^{\kappa}$ is countably compact for every $\kappa<\mu$ but $\CL(X)$ is not pseudocompact.
\end{theorem}
\begin{proof}
Let $(C_n: n \in \omega)$ be a sequence of nonempty pairwise disjoint finite subsets of $\omega$ such that $|C_{n+1}|\geq |C_n|$ for every $n \in \omega$ and such that $|C_n|\rightarrow \infty$. Let $\mathcal G=\{g \in \omega^\omega: g(n) \in C_n\}$, enumerating it as $\mathcal G=\{g_\alpha: \alpha<\mathfrak c\}$, and, for every $\alpha<\mathfrak c$, let $\mathcal G_\alpha=\{g_\beta: \beta<\alpha\}$. We enumerate $\mathcal F=\bigcup\{(\lambda^{\mathfrak \kappa})^\omega:\kappa \text{ is an infinite cardinal and } \kappa<\mu\}=\{f_\alpha: \alpha<\lambda\}$ such as for every $f \in \mathcal F$, $|\{\alpha<\mathfrak \lambda: f_\alpha=f\}|=\lambda$. This is possible since $\lambda^{\kappa}\leq\lambda$ for every infinite cardinal $\kappa< \mu$ and because there are at most $\mu\leq \lambda$ cardinals below $\mu$. For each $\alpha<\lambda$, let $\kappa_\alpha=\operatorname{dom}
f_\alpha(n)$ ($\kappa_\alpha$ doesn't depend on $n$)  and $\zeta_\alpha=\{f_\alpha(n)(\beta)+1: n \in \omega, \beta<\kappa_\alpha\}$. Notice that $\zeta_\alpha<\lambda$ since by hypothesis, $\operatorname{cf} (\lambda)\geq \mu>\mathfrak \kappa_\alpha, \omega$.

Recursively  we will define, for every $\alpha<\lambda$, ordinals $\delta_\alpha$ and $\epsilon_\alpha$,  sets $X_\alpha=\{x_\xi: \delta_\alpha\leq\xi<\epsilon_\alpha\}\subseteq \beta\omega$ and $Y_\alpha\subseteq \beta \omega$, sequences $\hat f_\alpha:\omega \to \left(\bigcup_{\beta< \alpha} X_\beta\right)^{\kappa_\alpha}$, and collections of free ultrafilters $P_\alpha$ satisfying:

\begin{enumerate}
	\item $\delta_0=0$, $\epsilon_0=\omega$ and $x_n=n$ for every $n<\omega$,
	
	\item $\delta_\alpha=\sup\{\epsilon_\beta: \beta<\alpha\}$ for every $\alpha<\lambda$, 
	
	\item the sequence $(\epsilon_\alpha: \alpha<\lambda)$ is strictly increasing,
	
	\item $0<|\epsilon_\alpha\setminus \delta_\alpha|\leq\mathfrak \kappa_\alpha$ for every $\alpha<\lambda$,
	
	\item $x_\xi\neq x_{\xi'}$ whenever $\xi\neq \xi'$,
	
	\item $(\bigcup_{\beta\leq\alpha}X_\alpha)\cap(\bigcup_{\beta\leq \alpha} Y_\alpha)=\emptyset$ for every $\alpha<\lambda$,
	
	\item $\hat f_\alpha(n)(\xi)=x_{f_\alpha(n)(\xi)}$ for every $\xi < \kappa_\alpha$,  $n < \omega$ and $\alpha<\lambda$ when $\zeta_\alpha<\sup\{\epsilon_\beta:\beta<\alpha\}$,

\item $\hat f_\alpha$ has an accumulation
 point in $\left(\bigcup_{\beta\leq \alpha} X_\beta\right)^{\kappa_\alpha}$
 when $\zeta_\alpha<\sup\{\epsilon_\beta:\beta<\alpha\}$,

\end{enumerate}Also, if $\lambda=\mathfrak c$:	
\begin{enumerate}
		  \setcounter{enumi}{8}
	\item $|Y_\alpha|\leq |\alpha|.|\epsilon_\alpha|$ for every $\alpha<\mathfrak c$, 

	\item $P_\alpha=\{p \in \omega^*: \exists g \in \mathcal G_\alpha\,(p-\lim g \in \bigcup_{\beta \leq \alpha} X_\alpha)\}$ for every $\alpha<\mathfrak c$, and
	
	\item $\forall \alpha< \mathfrak c \, \forall p \in P_\alpha \exists y \in Y_\alpha \exists h \in \mathcal G\, (p-\lim h=y)$.
\end{enumerate}
Else, if $\mathfrak c<\lambda\leq 2^\mathfrak c$:
\begin{enumerate}[label=(\arabic*)']
	\setcounter{enumi}{8}
	\item $|Y_\alpha|\leq \mathfrak c$ for every $\alpha<\lambda$, 
	\item $P_\alpha=\{p \in \omega^*:p \notin \bigcup_{\beta<\alpha}P_\beta\,\wedge\, \exists g \in \mathcal G\,(p-\lim g \in X_\alpha)\}$ for every $\alpha<\mathfrak c$, and
	
	\item $\forall \alpha< \mathfrak \lambda  \forall p \in P_\alpha \exists y \in Y_\alpha \exists h \in \mathcal G\, (p-\lim h=y)$.
\end{enumerate}

Clearly $\delta_0$, $\epsilon_0$ and $X_0$ are given by 1 and so $P_0=\emptyset$. Hence by taking $Y_0=\emptyset$ all the conditions hold for $\alpha=0$. Let $\alpha<\lambda$ and suppose that we have defined the $\epsilon_\beta, \delta_\beta$,  $X_\beta$, $Y_\beta$ and  $P_\beta$ for every $\beta<\alpha$. Let $\delta_\alpha$ and $\hat f_\alpha$ be defined by 2 and 7 respectively (if $\zeta_\alpha\geq\sup\{\epsilon_\beta:\beta<\alpha\}$, let $\hat f_\alpha$ be any sequence). Notice that  items 1-4 imply that 
$
\dot{\bigcup}_{\beta<\alpha}[\delta_\beta, \epsilon_\beta)=\sup\{\epsilon_\beta: \beta<\alpha\}<\lambda
$ 
because, if $\mu<\lambda$, then 
$$
\left|\sup\{\epsilon_\beta: \beta<\alpha\}\right|\leq \sum_{\beta<\alpha}\kappa_\alpha\leq\sum_{\beta<\alpha}\mu=\mu |\alpha|<\lambda;
$$
and otherwise, if $\mu=\lambda$ then the equality $\mu=\lambda=\mathfrak c$ holds, so $\mathfrak c\geq \operatorname{cf}(\mathfrak c)=\operatorname{cf}(\lambda)\geq\mu=\mathfrak c$, which implies $\mathfrak c$ is regular. Thus 
$
|\sup\{\epsilon_\beta: \beta<\alpha\}|\leq \sum_{\beta<\alpha}\kappa_\alpha<\mathfrak c
$. 
In any case, $\left|\bigcup_{\beta<\alpha}X_\beta \right| < \lambda$. Notice that, if $\lambda=\mathfrak c$, then by 9, 
$$
\left|\bigcup_{\beta<\alpha}Y_\beta\right|\leq\sum_{\beta<\alpha}|\beta||\epsilon_\beta|\leq |\alpha|\sum_{\beta<\alpha}|\epsilon_\beta|\leq |\alpha|\sup\{\epsilon_\beta: \beta<\alpha\}<\lambda.
$$

and if $\lambda>\mathfrak c$, then $\left|\bigcup_{\beta<\alpha}Y_\beta\right|\leq |\alpha|\mathfrak c<\lambda$. Thus, in any case, $|\bigcup_{\beta<\alpha}(X_\beta \cup Y_\beta)|<\lambda$.

Now we will consider two cases:

\smallskip

{\bf Case I:}  $\zeta_\alpha\geq \sup\{\epsilon_\beta: \beta<\alpha\}$ or the sequence $\hat f_\alpha$ has an accumulation point in $(\bigcup_{\beta<\alpha}X_\beta)^{\kappa_\alpha}$. In this case let $\epsilon_\alpha=\delta_\alpha+1$ and pick $x_{\delta_\alpha}\in \omega^*\setminus \bigcup_{\beta<\alpha}(X_\beta \cup Y_\beta)$ arbitrary. 

\smallskip

{\bf Case II:}
 $\zeta_\alpha<\sup\{\epsilon_\beta: \beta<\alpha\}$ and the sequence $\hat f_\alpha$ has no accumulation point in $(\bigcup_{\beta<\alpha}X_\beta)^{\kappa_\alpha}$. For this case let $Y=\bigcup_{\beta<\alpha}Y_\beta$. Since $|Y|<\lambda$, we have by  hypothesis and items 6  that the sequence $\hat f_\alpha$ has an accumulation point $(u_\delta)_{\delta \in \kappa_\alpha} \in (\beta\omega\setminus Y)^{\kappa_\alpha}$. Set $X_\alpha=\{u_\delta: \delta<\kappa_\alpha\}\setminus \bigcup_{\beta<\alpha}X_\beta$ and observe that  by  assumption of this case $0 < |X_\alpha|\leq \kappa_\alpha$ so, let $\epsilon_\alpha=\delta + |X_\alpha|$ (ordinal sum) and enumerate $X_\alpha$ as $\{x_\xi: \delta_\alpha \leq \xi < \epsilon_\alpha \}$.

\smallskip

From the two cases $\epsilon_\alpha$ and $X_\alpha$  are already defined satisfying  items 3-5 and 8. First, suppose $\lambda=\mathfrak c$. Hence $P_\alpha$ is given by 10, so  $|P_\alpha|\leq |\alpha||\epsilon_\alpha|$. Furthermore, by Lemma \ref{ContZp}, $ |Z_p|=\mathfrak c$, for each $p\in P_\alpha$, so for each such $p$ there is $h_p \in \mathcal G$ such that $p-\lim h_p \notin \bigcup_{\beta\leq \alpha}X_\beta$ (since the latter has cardinality $|\epsilon_{\alpha}|<\lambda=\mathfrak c$).  Then let $Y_\alpha=\{p-\lim h_p: p \in P_\alpha\}$ and observe that items 6, 9 and 11 holds. Now suppose $\mathfrak c<\lambda\leq 2^\mathfrak c$. Let $P_\alpha$ be as in 10', so, by 4, $|P_\alpha|\leq\left|\bigcup_{g \in \mathcal G}\{p \in \omega^*: p-\lim g \in X_\alpha\}\right|\leq \sum_{g \in \mathcal G}|X_\alpha|\leq \sum_{g \in \mathcal G}\kappa_\alpha\leq \mathfrak c$. For every $p \in P_\alpha$, by Lemma \ref{ContZp}, $ |Z_p|=\mathfrak c$, so there is $h_p \in \mathcal G$ such that $p-\lim h_p \notin X_\alpha$. Notice that $p-\lim h_p \notin \bigcup_{\alpha<\beta}X_\alpha$, or else, by letting $\alpha$ be the smallest ordinal such that $p-\lim h_p \in X_\alpha$, it follows that $p \in P_\alpha$, so $p \notin P_\beta$, a contradiction. Then by letting $Y_\alpha=\{p-\lim h_p: p \in P_\alpha\}$ and observe that items 6, , 9' and 11' holds.

\smallskip

Our promised  set will be  $X=\bigcup_{\alpha<\lambda} X_\alpha$.    By 1, $\omega \subseteq X$ and recall that items 1-5 say that $|X|=\lambda$.  Let $\kappa<\mu$. To prove that $X^\kappa$ is countably compact let  $h:\omega\rightarrow X^{\kappa}$. By 7 and the definition of $\mathcal F$ there is $\alpha_0<\lambda$  such that $\hat f_{\alpha_0}=h$.  Even more, the same statements guarantee that there is $\alpha_1<\lambda$ such that $\zeta_{\alpha_0} < \epsilon_{\alpha_1}$ and $f_{\alpha_0}=f_{\alpha_1}= h$. Then, for $\alpha=\alpha_1$ it is true that $\epsilon_\alpha>\zeta_\alpha$, $\hat f_{\alpha}=h$ and by  8, the sequence $h$ has an accumulation point on $X^{\kappa}$. This proves that $X^{\kappa}$ is countably compact.

Now we will show that the sequence $(C_n: n \in \omega)$ has no accumulation point on $\CL(X)$. Suppose $F$ is  an accumulation point for that sequence, so there exists a free ultrafilter $p$ such that $p-\lim C_n=F$. By Lemma \ref{ContZp}, $Z_p\subseteq F\subseteq X$  (where $Z_p=\{p-\lim g: g \in \mathcal G\}$). Let $\alpha$ be the first ordinal such that $Z_p\cap X_\alpha \neq \emptyset$.

In case $\lambda=\mathfrak c$, there exists $\beta<\lambda$ such that $p-\lim g_\beta \in Z_p\cap X_\alpha$. Let $\gamma=\max\{\alpha, \beta\}+1$, then by 10, $p \in P_\gamma$. Therefore $Z_p \cap Y_\gamma \neq \emptyset$, which implies $Z_p \setminus X\neq \emptyset$, a contradiction.

Finally, if $\mathfrak c<\lambda\leq 2^\mathfrak c$, there exists $g \in \mathcal G$ such that $p-\lim g\in Z_p\cap X_\alpha$. By 10', $p \in P_\alpha$, therefore $Z_p \cap Y_\alpha \neq \emptyset$, which implies $Z_p \setminus X\neq \emptyset$, a contradiction.
\end{proof}

Now we concretely explore the existence of spaces whose large products are countably compact whose hyperspace is not pseudocompact.

\begin{example}\label{t1}There exists $X\subseteq \beta \omega$ of cardinality $\mathfrak c$ such that $\omega\subseteq X$ and $X^{\mathfrak \kappa}$ is countably compact for every $\kappa<\mathfrak t$ but $ \CL(X)$ is not pseudocompact.
\end{example}

\begin{example}\label{exampleT}Suppose $\mathfrak t<\mathfrak c$ and $2^\mathfrak t=\mathfrak c$. Then there exists $X\subseteq \beta \omega$ of cardinality $2^{\mathfrak t}$ such that $ \CL(X)$ is not pseudocompact, $\omega\subseteq X$ and $X^{\mathfrak t}$ is countably compact.
\end{example}

\begin{proof}
	The two constructions are very similar, so we will construct both at the same time.
	
	For the first example, We apply Theorem \ref{smallbig} by using $\mu=\mathfrak t$, $\lambda=\mathfrak c$. It is clear that $\omega_1\leq\mathfrak t\leq \mathfrak c\leq \mathfrak c\leq 2^\mathfrak c$, and it is well known (see \cite{CombinatorialBlass}) that for every infinite cardinal $\kappa<\mathfrak t$, $2^\kappa=\mathfrak c$, so $\mathfrak c^{<\mathfrak t}=\mathfrak c$ and, by K\"{o}nig's Lemma, $\operatorname{cf}(\mathfrak c)\geq \mathfrak t$. For the second example, we apply Theorem \ref{smallbig} by letting $\mu=\mathfrak t^+\leq \mathfrak c$, $\lambda=2^\mathfrak t=\mathfrak c$.
	
	 Suppose $\kappa<\mu$, $|Y|\in [\omega^*]^{<\mathfrak c}$ and $F:\omega\rightarrow (\beta \omega\setminus Y)^\kappa$. We must show that $F$ has an accumulation point in $(\beta \omega\setminus Y)^\kappa$. For every $\delta<\kappa$, let $g_\delta=\pi_\delta\circ F$, where $\pi_\delta$ is the projection on the $\delta$-th coordinate.
	
	We recursively define a $\supseteq^*$ tower $(A_\delta: \delta<\kappa)$ such that for every $\delta<\kappa_\alpha$, $g_\delta|A_{\delta+1}$ is either injective or constant, $g_\delta[A_{\delta+1}]$ is a discrete subspace of $\beta \omega$, and that $Y\cap\cl g_\delta[A_{\delta+1}]=\emptyset$. To see we can carry on such a recursion, let $A_0=\omega$. Suppose we have defined $A_\beta$ for every $\beta<\delta$. If $\delta=\beta+1$ for some $\beta$, let $f=g_\beta$. Let $X\subseteq A_\delta$ be such that $f|X$ is either constant or injective and $f[X]$ is discrete. If $f|X$ is constant, let $A_{\delta+1}=X$. Else, let $\mathcal A$ be an almost disjoint family of cardinality $\mathfrak c$ on $X$. Each $y \in Y$ is in at most one element of $\{\cl f[A]: A \in \mathcal A\}$, for if $A, B \in \mathcal A$ are distinct and $y \in \cl f[A]\cap \cl f[B]$, since $y \notin f[A]\cup f[B]$ it follows that $y \in \cl(f[A]\setminus f[B])\cap \cl(f[B]\setminus f[A]),$
	 then, since both sets inside of the closure operator are countable, without loss of generality: $$(f[A]\setminus f[B])\cap \cl(f[B]\setminus f[A])\neq \emptyset,$$ then, since $f[X]$ is discrete, $(f[A]\setminus f[B])\cap (f[B]\setminus f[A])\neq \emptyset,$ a contradiction. Since $|\{\cl f[A]: A \in \mathcal A\}|=\mathfrak c$ and $|Y|<\mathfrak c$, there exists $A \in \mathcal A$ such that $y \notin \cl f[A]$ for every $y \in Y$. Let $A_{\delta+1}$ be such an $A$. Finally, if $\delta$ is limit, simply let $A_\delta$ be a pseudo-intersection of $(A_\beta: \beta<\delta)$. Finally, let $r$ be a free ultrafilter containing $\{A_\delta: \delta<\kappa\}$. Then $r-\lim g_\delta \in \cl g_\delta[A_{\delta+1}]$, so $r-\lim g_\delta \notin Y$. Therefore, $r-\lim F \in (\beta \omega \setminus Y)^\kappa$.
\end{proof}

Recall that by starting with GCH and adding at least $\omega_2$ Cohen reals, we get a model where $\mathfrak t=\omega_1$ and $2^{\omega_1}=\mathfrak c$ (\cite{CombinatorialBlass}, \cite{kunen1980set}, \cite{kunen2011set}).

\smallskip

In some sense, the theorem below is a strengthening of Theorem \ref{t1}. The only difference between them is the cardinality of the space.
\begin{example}\label{exampleH}There exists $X\subseteq \beta \omega$ of cardinality $2^{\mathfrak h}$ such that $ \CL(X)$ is not pseudocompact, $\omega\subseteq X$ and $X^{\kappa}$ is countably compact for every $\kappa<\mathfrak h$.
\end{example}

\begin{proof}
		We apply Theorem \ref{smallbig} by letting $\mu=\mathfrak h\leq \mathfrak c$, $\lambda=2^\mathfrak h$. It is clear that $\omega_1\leq \mathfrak h\leq \mathfrak c\leq 2^{\mathfrak h}\leq 2^{\mathfrak c}$. It is clear that $\lambda^{<\mathfrak h}=\lambda$, and $\operatorname{cf}(\lambda)> \mathfrak h$ by K\"{o}nig's Lemma. It remains to see that for every $Y\in [\omega^*]^{<\lambda}$, $\kappa\leq \mathfrak h$ and $F:\omega\rightarrow (\beta \omega\setminus Y)^\kappa$, $F$ has an accumulation point in $(\beta \omega\setminus Y)^\kappa$. For every $\delta<\kappa$, let $g_\delta=\pi_\delta\circ F$, where $\pi_\delta$ is the projection on the $\delta$-th coordinate.
	
	For each $\delta<\mathfrak \kappa$, let
$$
U_\delta=\{A \in [\omega]^\omega: \exists r \in \omega \text{ for which } g_\delta[A\setminus r] \text{ is a discrete subset of } \beta \omega,
$$
$$
\text{ and } g_\delta|(A\setminus r) \text{ is either a constant or an injective function}\}.
$$
	
	It is clear that $U_\delta$ is open dense in $[\omega]^\omega$. Since $\kappa_\alpha<\mathfrak h$, let $A \in \bigcap_{\delta<\kappa_\alpha}U_\delta$. Let $A^*\subseteq \beta A$ be the set of all free ultrafilters over $A$. Notice that if $g:A\rightarrow \beta \omega$ and there exists $r \in \omega$ such that $g|[A\setminus r]$ is injective, then whenever $p, q \in A^*$ are distinct, $p-\lim g\neq q-\lim g$. Define $B_\delta=\{p \in A^*: p-\lim g_\delta|A \notin Y\}$. It follows, that $|A^*\setminus B_\delta|\leq |Y|$, therefore $\bigcap_{\delta<\kappa_\alpha}B_\delta$ is nonempty. Let $r$ be an ultrafilter in this intersection. It follows that $r-\lim F \in (\beta \omega\setminus Y)^\kappa$.
	
\end{proof}

\section{Relating $(\kappa, \omega^*)$-pseudocompactness with the Pseudocompactness of $\CL(X)$}

Recall that if $X$ is $T_2$, then $\CL(X)$ is $T_1$ and that if $D\subseteq X$ is dense and $X$ is $T_1$, then $[D]^{<\omega}\setminus \{\emptyset\}\subseteq \CL(X)$ is dense. The proof of the lemma below is straightforward and the proof is left to the reader.

\begin{lemma}\label{l4.1}
If $D\subseteq X$ is an open discrete dense subset of  the space $X$ and $\kappa$ is a cardinal, then $D^\kappa$ is relatively countably compact in $X$ if and only if $X$ is $(\kappa, \omega^*)-$pseudocompact.
\end{lemma}

\begin{theorem}\label{ispseudo}Suppose $X$ is $T_2$. Let $D\subseteq X$ be a dense subset of $X$. If $D^\mathfrak c$ is relatively countably compact in $X^\mathfrak c$ then $[D]^{<\omega}\setminus \{\emptyset\}$ is relatively countably compact in $\CL(X)$ and $\CL(X)$ is pseudocompact.
\end{theorem}

\begin{proof}Let $f:\omega\rightarrow [D]^{<\omega}\setminus \{\emptyset\}$ be a sequence of nonempty finite subsets of $D$ and let $F_n=f(n)$. Let $\mathcal G=\{g \in D^\omega: g(n) \in F_n\}$. It follows that $|\mathcal G|\leq \mathfrak c$. Since $D^\mathfrak c$ is relatively countably compact in $X^\mathfrak c$, there exists a free ultrafilter $p$ such that for every $g \in \mathcal G$, there exists a $p-\lim g$ in $X$. Let $Z=\{p-\lim g: g \in \mathcal G\}$ and let $F=\cl Z$. We claim that $p-\lim f=F$. So let $U$ be an open neighborhood of $F$. We must verify that $\{n \in \omega: F_n \in W\} \in p$. It suffices to prove this claim for $U$ subbasic. So let $W\subseteq \omega$ be open. If $U=\{K \in \CL(X): K\subseteq W\}$, suppose by contradiction that $\{n \in \omega: F_n \in W\}\notin p$. Then $A=\{n \in \omega: F_n \setminus W\neq \emptyset\} \in p$. Let $g \in \mathcal G$ be such that $g(n) \in F_n \setminus W$ for every $n \in A$. Then $p-\lim g \in X\setminus W$, a contradiction since $Z\subseteq W$.
	
	Now suppose $U=\{K \in \CL(X): K\cap W \neq \emptyset\}$. Since $F\cap W\neq \emptyset$, there exists $x \in Z\cap W$ and there exists $g \in \mathcal G$ such that $x=p-\lim g$. It follows that $p \ni \{n \in \omega: g(n) \in W\}\subseteq\{n \in \omega: F_n \cap W\neq \emptyset\}$. This proves that $[D]^{<\omega}\setminus \{\emptyset\}$ is relatively countably compact. To see that $\CL(X)$ is pseudocompact, recall that for every space $Y$, if there exists a dense subset of $Y$ that is relatively countably compact in $Y$, then $Y$ is pseudocompact.
	
\end{proof}

\begin{cor}\label{cor36}
	Suppose $\omega\subset X\subset \beta \omega$ and suppose $X$ is $(\mathfrak c, \omega^*)$-pseudocompact (equivalently, that $\omega^\mathfrak c$ is relatively countably compact in $X^\kappa$). Then $\CL(X)$ is pseudocompact.
\end{cor}

Recall that two ultrafilters $q_0, q_1$ over $\omega$ are said to be incomparable if for every bijection $f:\omega\rightarrow \omega$, the ultrafilter generated by $\{f[A]: A \in q_0\}$ is not $q_1$. Example \ref{ultrafilters} implies that, since it is consistent with the axioms of ZFC that there exists $\mathfrak c$ pairwise incomparable selective ultrafilters (\cite{CombinatorialBlass}), we may not weaken the cardinal $\mathfrak c$ on Corollary \ref{cor36}. The lemma below is well known, so we state it with no proof.

\begin{lemma} If $h_0, h_1 \in \omega ^\omega$ and $q_0, q_1$ are two incomparable selective ultrafilters such that $h_i$ is not $q_i$-equivalent to a constant
sequence, then, in $\beta \omega$, $q_0-\lim h_0\neq q_1-\lim h_1$.
\end{lemma}

\begin{cor} Given a ultrafilter $q$ denote by $S(q)=\{q-\lim h: h \in \omega^\omega\}\cap \omega^*$. Then if $q_1, q_2$ are incomparable selective ultrafilters, $S(q_1)\cap S(q_2)=\emptyset$.
\end{cor}

\begin{example}\label{ultrafilters} Assume that ${\mathfrak c}$ is regular and that there exists ${\mathfrak c}$ incomparable selective ultrafilters. Then there exists $X$, $\omega \subseteq X \subseteq \beta \omega$ such that $\omega ^\alpha$ is relatively countably compact in $X^\alpha$ for each $\alpha <{\mathfrak c}$ but $\CL(X)$ is not pseudocompact.
\end{example}

\begin{proof} Enumerate $\omega^\omega$ as $\{f_\alpha:\,\alpha <{\mathfrak c}\}$ where $f_0$ is the constant $0$ function. Let $C_n =[2^n, 2^{n+1}-1]$ for each $n\in \omega$ and ${\mathcal G}=\{g\in \omega^\omega:\forall n \in \omega\,(f(n)\in C_n)\}$. Fix a surjective function $i:\, {\mathfrak c} \longrightarrow {\mathfrak c} \times {\mathfrak c}$ such that $i(\mu)=(\xi_\mu,\eta_\mu)$ with $\xi_\mu \leq \mu$ for each $ \mu < {\mathfrak c}$.

Given a ultrafilter $p$, let $S_\alpha(p)=\{x \in \omega^*:\exists \beta \leq \alpha\,(x=p-\lim f_\beta)\}$,  $S(p)=\{p-\lim h: h \in \omega^\omega\}$ and $Z_p=\{p-\lim h: h \in \mathcal G\}$. Recursively, we will define for $\alpha<\mathfrak c$, selective ultrafilters $p_\alpha$, subsets of $\beta\omega$ $X_\alpha$, sets of free ultrafilters $P_\alpha= \{q_{(\alpha, \mu)}:\mu<\mathfrak c\}$ and $y_\alpha \in \beta \omega$ such that:

\begin{enumerate}
	\item $X_0=\omega \cup S_0(p_0)$, $X_\alpha=S_\alpha(p_\alpha)$ for $0<\alpha<\mathfrak c$,
	\item The family $\{X_\alpha: \alpha<\mathfrak c\}\cup\{\{y_\alpha: \alpha<\mathfrak c\}\}$ is pairwise disjoint,
	\item $\{q \in \omega^*: Z_q\cap X_\alpha \neq \emptyset\}\subset P_\alpha$, and
	\item $y_\alpha \in Z_{q_{i(\alpha)}}\setminus \left(\bigcup_{\beta\leq \alpha} X_\beta \cup \{y_\beta: \beta<\alpha\}\right)$.
	
\end{enumerate}

 Fix $p_0$ an arbitrary selective ultrafilter. Set $X_0=\omega \cup S_0(p_0)$. Since $|X_0|<\mathfrak c$ and $\{q\in \omega^*: Z_q \cap X_0=\emptyset\}=\bigcup_{g \in \mathcal G}\{q \in \omega^*:q-\lim g \in X_0\}$, the latter has cardinality $\leq \mathfrak c$, therefore we can enumerate ultrafilters $P_0=\{ q_{(0,\mu)}:\, \mu < {\mathfrak c}\}$ containing this set. Set $y_0 \in Z_{q_{i(0)}} \setminus X_0$, which is possible by Lemma \ref{ContZp}.

Suppose that $X_\beta$, ${P}_\beta =\{q_{(\beta , \mu)}:\, \mu <{\mathfrak c}\}$, $p_\beta $ and $\{y_\beta :\, \beta < \alpha \}$ have been defined for every $\beta<\alpha$ for some $\alpha \in [1, \mathfrak c)$.

Since the $S(p)$'s are pairwise disjoint for incomparable selective ultrafilters and $|\{y_\beta :\, \beta < \alpha \}|<\mathfrak c$, there exists a selective ultrafilter $p\notin \{p_\beta: \beta<\alpha\}$ such that $S(p) \cap (\bigcup_{\beta < \alpha} X_\beta \cup \{y_\beta :\, \beta < \alpha \})=\emptyset$. Denote such $p$ as $p_\alpha$ and set $X_\alpha =S_\alpha (p_\alpha)$. As before, define ${P}_\alpha\supset\{q \in \omega^*:Z_q \cap X_\alpha \neq \emptyset\}$ enumerating it as $\{q_{(\alpha, \mu)}:\, \mu <{\mathfrak c}\} \supseteq {P}_\alpha$. Fix $y_\alpha \in Z_{q_{i(\alpha)}}\setminus (\bigcup_{\beta \leq \alpha} X_\beta \cup \{y_\beta :\, \beta < \alpha \})$. This ends the construction.

{\em Claim 1}. $\omega^\mu$ is relatively countably compact in $X^\mu$ for each $\mu <{\mathfrak c}$.

Let $(g_\xi :\, \xi < \mu) $ be a collection of $\mu$ sequences into $\omega$. Since ${\mathfrak c}$ is regular, there exists $\alpha < {\mathfrak c}$ such that
$\{g_\xi :\, \xi < \mu\} \subseteq \{f_\beta :\, \beta < \alpha \}$. Then each $g_\xi $ has $p_\alpha$-limit in $S_\alpha (p_\alpha) \cup \omega \subseteq X$.

\smallskip

{\em Claim 2}. $\{ C_n:\, n \in \omega \}$ witnesses that $\CL(X)$ is not pseudocompact.

Suppose that $\{ C_n:\, n \in \omega \}$ has an accumulation point. Then by Lemma \ref{ContZp}, there exists $q \in \omega^*$ such that $Z_q \subseteq X$. Let $\alpha < {\mathfrak c}$ such that $Z_q \cap X_\alpha \neq \emptyset$.
Then $q = q_{(\alpha, \mu)}$ for some $\mu <{\mathfrak c}$. Let $\theta < {\mathfrak c}$ such that $i(\theta)=(\alpha, \mu)$. Then $q=q_{i(\theta)}$ and $y_\theta \in Z_q \setminus X$, a contradiction.
Therefore, $\CL(X)$ is not pseudocompact.
\end{proof}
Theorem \ref{ispseudo} may also be used to give a simpler proof of Theorem 3.2 of \cite{PseudoHyper} as we will see, that states that under $\mathfrak p=\mathfrak c$, $\CL(\Psi(\mathcal A))$ is pseudocompact for every mad family $\mathcal A$. S. Shelah and M. Malliaris have recently proved that $\mathfrak p=\mathfrak t$ (\cite{PisT}), however, we avoid the complicated model theoretic proof of $\mathfrak p =\mathfrak t$.

\begin{lemma}
	
For every mad family $\mathcal A$, $\omega^\mathfrak t$ is relatively countably compact in $\Psi(\mathcal A)^{\mathfrak t}$, thus, $\Psi(\mathcal A)^\mathfrak t$ is $(\mathfrak t, \omega^*)-$pseudocompact.
\end{lemma}

\begin{proof}Let $F \in \omega^\mathfrak t$. We will see that there exists a free ultrafilter $p$ such that $p-\lim F$ exists in $\Psi(\mathcal A)^{\mathfrak t}$. For each $\delta<\mathfrak c$, let $g_\alpha=\pi_\alpha\circ F$, where $\pi_\alpha$ is the $\alpha$-th projection.
		
			Recursively, we define a $\supseteq^*$-tower $(A_\alpha: \alpha<\mathfrak t)$ such that for each $\alpha<\mathfrak c$ the $g_\alpha|A_{\alpha+1}$ converges as follows: Let $A_0=\omega$. Suppose we have defined $A_\alpha$. There exists $A \in \mathcal A$ such that $|A\cap g_\alpha[A_\alpha]|=\omega$.  Let $A_{\alpha+1}=g_\alpha^{-1}[A\cap g_\alpha[A_\alpha]]$, so $g_\alpha|A_{\alpha+1}$ converges. For the limit step $\alpha$, let $A_\alpha$ be a pseudo-intersection of $(A_\beta: \beta<\alpha)$. Let $p$ be a free ultrafilter containing each element of the tower. Notice that since $A_{\alpha+1} \in p$ for each $\alpha$, $p-\lim g_\alpha$ exists, therefore $p-\lim F$ exists.
			
\end{proof}
By combining the above lemma and Theorem \ref{ispseudo}, we get Theorem 3.2 of \cite{PseudoHyper}, which states that under $\mathfrak p=\mathfrak c$, $\CL(\Psi(\mathcal A))$ is pseudocompact for every mad family $\mathcal A$.

\begin{cor}Assuming $\mathfrak t=\mathfrak c$, $\CL(\Psi(\mathcal A))$ is pseudocompact for every mad family $\mathcal A$.\end{cor}
\section{Relating the Pseudocompactness of $\CL(X)$ with $(\kappa, \omega^*)$-pseudocompactness}
Now we investigate what happens if $\CL(X)$ is pseudocompact.

\begin{theorem}\label{EverySmallKappa} Let $X$ be a subspace of $\beta \omega$ that contains $\omega$. If $\CL(X)$ is pseudocompact, then for every $\kappa<\mathfrak h$, $\omega^\kappa$ is relatively countably compact in $X^\kappa$, therefore the latter is $(\kappa, \omega^*)$-pseudocompact.
\end{theorem}

\begin{proof}
	For each $\alpha<\kappa$, let $(x_\alpha(n): n \in \omega)$ be a sequence of natural numbers. Since $\kappa<\mathfrak h$ and for each $\alpha<\kappa$, $$D_\alpha=\{A \in [\omega]^\omega: x_\alpha|A \text{ is eventually strictly growing or } \text{eventually} \text{ constant}\}$$ is open dense, by taking a subsequence we may suppose that each $x_\alpha$ is either eventually strictly increasing or eventually constant, and by removing the eventually constant sequences, we may suppose that each sequence is eventually strictly increasing.
	
	Since $\mathfrak b\geq \mathfrak h$, let $a=(a(n): n \in \omega)$ be a strictly increasing sequence such that $a>^*x_\alpha$ for every $\alpha<\kappa$. For each $m \in \omega$, let $$I_m=\{\alpha<\kappa: \forall n\geq m \,(a(n)>x_\alpha(n)) \text{ and } x_\alpha|(\omega\setminus m) \text{ is strictly increasing}\}.$$
	
	Notice that $I_m$ is a growing sequence of subsets of $\kappa$. Recursively, we define a strictly increasing sequence $n_k$ of natural numbers satisfying: $n_0=0$ and, for every $k$, $\forall \alpha \in I_{n_k}\, a(n_k)<x_\alpha(n_{k+1})<a(n_{k+1}).$
	
	In order to do it, defined $n_0, \dots, n_k$, let $n_{k+1}=2a(n_k)+1$. Given $ \alpha \in I_{n_k}$, $a(k)<x_\alpha(n_{k+1})$ since $x_\alpha|(\omega\setminus n_k)$ is strictly increasing and $n_{k+1}>n_k$, it follows that:
	$$a(n_k)<a(n_k)+1\leq x_\alpha(2a(n_k)+1)=x_\alpha(n_{k+1})<a(n_{k+1}).$$
	
	Let $C_0$ be the discrete interval $[0, a_0]$, and for each $k$, let  $C_{k+1}=[a_{n_k}+1, a_{n_k+1}]$. Then the $C_k$'s are nonempty, finite and $\omega$ is their disjoint union. Since $\CL(X)$ is pseudocompact, there exists $p \in \omega^*$ such that the sequence $C_k$ has a $p$-limit. There exists $m \in \omega$ such that for each $k\geq m$, $x_\alpha(n_k) \in C_k$. Therefore, by Lemma $\ref{ContZp}$, $p-\lim x_\alpha(n_k) \in X$ and we are done.
\end{proof}

Theorem \ref{EverySmallKappa}, along with Theorem 2.3 along of \cite{NYAS:NYAS22} (which states that for a space $X$, $X^\omega$ is pseudocompact if, and only if there exists $\kappa\geq \omega$ such that $X$ is $(\kappa, \omega^*)-$pseudocompact), implies the following;

\begin{cor}
	For every subspace of $\beta \omega$ containing $\omega$, if $\CL(X)$ is pseudocompact then all powers of $X$ are pseudocompact.
\end{cor}

The Theorem above shows that if $\omega\subseteq X\subseteq \beta \omega$ is such that  $\CL(X)$ is pseudocompact, then $X$ is $(\kappa, \omega^*)$-pseudocompact for every $\kappa<\mathfrak h$. However, we can't increase $\kappa$ too much, as we shall see, we can't conclude $(\mathfrak b, \omega^*)$-pseudocompactness.

The following definition will be useful in some of the following constructions.

\begin{definition}Suppose $C$ is a sequence of finite sets. Given $X \in [\omega]^\omega$, a \textit{nice dissection} of $C$ over $X$ is a pair $(U, D)$ of sequences such that $U|X$ is increasing, $D|X$ is pairwise disjoint and for every $n \in X$, $U\cap D=\emptyset$ and $C(n)=U(n)\cup D(n)$.
\end{definition}

\begin{lemma}The following claims holds:\label{secc}
	\begin{enumerate}[label=\alph*)]
		\item For every sequence of finite sets $C$ and for every $Y \in [\omega]^\omega$, there exists $X \in [Y]^\omega$ such that $C$ admits a nice dissection over $X$.
		\item Suppose $K$ is a compactification of a discrete space $D$. Suppose $D\subseteq X\subseteq K$ is such that every sequence of pairwise disjoint finite nonempty subsets of $D$ has an accumulation point in $\CL(X)$. Then $\CL(X)$ is pseudocompact.
	\end{enumerate}

\end{lemma}

\begin{proof}
	
	For a), let $Y$ and $C$ be given. Recursively, we choose $x_n \in Y$ and a decreasing sequence $J_n\in [Y]^{\omega}$ such that:
	
	\begin{enumerate}
		\item $J_{n+1}\cap (x_n+1)=\emptyset$,
		\item $x_0 \in Y$, $x_{n+1} \in J_n$ for each $n \in \omega$, and
		\item $\forall t \in C(x_n)[(\forall j \in J_n\, (t \in C(j))) \vee (\forall j \in J_n\, (t \notin C(j))]$.
	\end{enumerate}

This is possible since each $C(n)$ is finite. Then $C$ admits a nice dissection over $X=\{x_n: n \in \omega\}$ by setting $U(x_n)=\{t \in C(x_n): \forall j \in J_n\, (t \in C(j))\}$ and $D(x_n)=\{t \in C(x_n): \forall j \in J_n\, (t \notin C(j))\}$.

	For b), since $E=[D]^{<\omega}\setminus \{\emptyset\}$ is dense in $\CL(X)$, it suffices to show that every sequence $(A_n: n \in \omega)$ of elements of $E$ has an accumulation point in $\CL(X)$. By a), let $J \in [\omega]^\omega$ be such that $A$ admits a nice dissection $(U, D)$ over $J$. If there exists an infinite $J'\subseteq J$ such that $D(l)=\emptyset$ for every $l \in J'$, it follows that for each $l \in J$, $A_l=U_l$, so $(A_l: l \in J')$ converges to $E=\cl\left(\bigcup_{l \in J'}U_l\right)$ and we are done. Else, there exists an infinite $J'\subseteq J$ such that $D(l)\neq \emptyset$ for every $l \in J'$. By hypothesis, $(D(l): l \in J')$ has an accumulation point $F \in \CL(X)$, therefore $(A_l: l \in J')$ has $E\cup F$ as an accumulation point, where $E$ is as above.
\end{proof}

\begin{theorem}\label{NonPseudo} Let $\theta, \kappa\leq \mathfrak c$ be infinite cardinals. We give $\kappa$ the discrete topology. Suppose that there exists $\mathcal A \subseteq [\kappa]^{\omega}$ such that $|\mathcal A|\leq\theta$ and that for every sequence $\mathcal C=(\mathcal C_n: n \in \omega)$ of finite nonempty pairwise disjoint subsets of $\kappa$ there exists $E \in [\omega]^\omega$ and $B \in \mathcal A$ such that $|B\cap \bigcup_{n \in E}\mathcal C_n|<\omega.$
	
	Then there exists $X$ such that $\kappa\subseteq X\subseteq \beta \kappa$, $\CL(X)$ is pseudocompact and $X$ is not $(\theta, \omega^*)$-pseudocompact.
\end{theorem}

\begin{proof} Enumerate all possible sequences of pairwise disjoint finite nonempty subsets of $\kappa$ as $( {\mathcal C}_\alpha :\, 0<\alpha <{\mathfrak c})$. For every $\alpha$ such that $0<\alpha<\mathfrak c$, let $G_\alpha=\{f \in \kappa^\omega: f(n) \in \mathcal C_\alpha(n)\}$. For each $A \in \mathcal A$, let $f_A:\omega\rightarrow A$ be 1-1 and onto.
	
	Let $X_0=\kappa$. Recursively for $0<\alpha<\mathfrak c$, we define $X_\alpha\subseteq \beta\kappa$, a free ultrafilter $q_\alpha$ on $\omega$, $P_\alpha\subseteq \beta \omega $ and $Y_\alpha\subseteq \beta \kappa$ satisfying:

	\begin{enumerate}
		
		\item $X_\alpha=\{q_\alpha-\lim g: g \in G_\alpha\}$,

		\item $P_\alpha=\{p \in \omega^*:p \notin \bigcup_{0<\beta<\alpha}P_\beta\,\wedge\, \exists A \in \mathcal A\,(p-\lim f_A \in X_\alpha)\}$,
		\item $(\bigcup_{\beta\leq\alpha}X_\beta)\cap(\bigcup_{0<\beta\leq \alpha} Y_\beta)=\emptyset$ (for every $0<\alpha<\mathfrak c$),
		\item $\forall p \in P_\alpha \,\exists B\in \mathcal A\, (p-\lim f_B\in Y_\alpha)$, and
		\item $|Y_\alpha|\leq \mathfrak c$.
	\end{enumerate}
	
	Suppose we have defined $X_\alpha, q_\alpha, P_\alpha, Y_\alpha$ for every $\alpha$ such that $0<\alpha<\beta$ for some $\beta<\mathfrak c$.
	
 There exists $E \in [\omega]^\omega$ and $B \in \mathcal A$ such that $|B\cap \bigcup_{n \in E}\mathcal C_\alpha(n)|<\omega$. Notice that given $g \in G_\alpha$ and $z \in \beta \kappa$, there exists at most one $r$ free ultrafilter such that $r-\lim g=z$, therefore, $Y=\{r \in \omega^*: \exists g \in \mathcal G_\alpha\,(r-\lim g \in \bigcup _{0<\beta < \alpha } Y_\beta)\}$ has cardinality at most $\mathfrak c$. Let $q_\alpha \in \omega^*\setminus Y$ be such that $E \in q_\alpha$.
	
	Let $X_\alpha$ as in 1 and $P_\alpha$ as in 2. Let $Y_\alpha=\{p-\lim f_B: p \in P_\alpha\}$. 1, 2, 4 and 5 clearly holds. To see that 3 holds, we first check that $Y_\alpha \cap X_\alpha=\emptyset$. Given $g \in G_\alpha$ and $p \in P_\alpha$, it follows that $g[E] \in q_\alpha-\lim g$ and  $B \in p-\lim f_B$. Since $g[E]\cap B$ is finite and $g$ is injective, it follows that $p-\lim f_B\neq q_\alpha -\lim g$. This proves $Y_\alpha \cap X_\alpha= \emptyset$. If $0<\gamma<\alpha$ and $Y_\alpha \cap X_\gamma \neq \emptyset$, then $Y_\alpha \cap \bigcup_{0<\beta\leq \alpha} Y_\beta\neq \emptyset$, a contradiction. This finishes the recursive construction.
	
	 Let $X=\bigcup_{\alpha <{\mathfrak c}}X_\alpha$. By construction, $\{q_\alpha-\lim g: g \in G_\alpha\} \subseteq X$, so by Lemma \ref{ContZp}, $q_\alpha-\lim {\mathcal C}_\alpha\in\CL(X)$, so by Lemma \ref{secc} b), $\CL(X)$ is pseudocompact.
	To see that $X$ is not $(\theta, \omega^*)-$pseudocompact, it suffices to show that for every free ultrafilter $p$ there exists $B \in \mathcal A$ such that $p-\lim f_B \notin X$, so fix a free ultrafilter $p$. If for some $f \in \mathcal A$, $p-\lim f \in X$, let $\alpha$ be the least ordinal for which $p-\lim f\in X_\alpha$. Then $p \notin P_\beta$ for $\beta < \alpha$, thus,
	$p \in P_\alpha$. Therefore, by 4, there exists $h \in \mathcal A$ such that $p-\lim h\in Y_\alpha$. Since $X \cap Y_\alpha = \emptyset$ we are done.
\end{proof}

\begin{example}\label{ExistsB} There exists a space $X$, $\omega \subseteq X \subseteq \beta \omega$ such that $\CL(X)$ is pseudocompact
	and $X$ is not $(\mathfrak b , \omega^*)$-pseudocompact.
\end{example}

\begin{proof}Let $\mathcal B$ be a unbounded family of cardinality $\mathfrak b$ such that each $g \in \mathcal B$ is strictly increasing. We apply Theorem \ref{NonPseudo} with $\kappa=\omega$ and $\theta=\mathfrak b$, where $\mathcal A=\{g[\omega]: g \in\mathcal B\}$.
	
	Let $\mathcal C$ be a sequence of pairwise disjoint finite nonempty subsets of $\omega$. Let $f:\omega\rightarrow \omega$ be such that $f(m)=1+\max \bigcup_{k\leq 2m}\mathcal C_k$ for each $m \in \omega$. There exists $g \in \mathcal B$ such that $N=\{m \in \omega: g(m)\geq f(m)\}$ is infinite. We claim that $E=\{n \in \omega:\mathcal C_n\cap  g[\omega]=\emptyset\}$ is infinite. Notice that if $m\in N$, then for every $p\geq m$, $g(p) \notin \bigcup_{k\leq 2m}\mathcal C_k$. Therefore at most $m$ elements of $\{\mathcal C_0, \dots, \mathcal C_{2m}\}$ have elements of $g[\omega]$, so $|E|\geq 2m+1-m=m+1$. Since $m \in N$ is arbitrary, $E$ is infinite. Then $g[\omega]\cap \bigcup_{n \in E}\mathcal C_n=\emptyset$ and the proof is complete.
\end{proof}

It is consistent that $\mathfrak b=\omega_1<\mathfrak c$. It is also consistent that $\omega_1<\mathfrak t$. Therefore, Theorems \ref{EverySmallKappa} and \ref{ExistsB} imply the following corollary:

\begin{cor}The existence of a space $X\subseteq \beta \omega$ containing $\omega$ such that $\CL(X)$ is pseudocompact and $X$ is not $(\omega_1, \omega^*)$-pseudocompact is independent of the axioms of ZFC+$\neg$CH.
\end{cor}
 However, the same doesn't apply if we consider subspaces of $\beta\omega_1$. Using the techniques above, we show that the answer is affirmative if we consider $\beta \kappa$ for $\omega_1\leq \kappa\leq \mathfrak c$ instead of $\beta \omega$.

\begin{example} Suppose $\omega_1\leq \kappa\leq \mathfrak c$. Then there exists a space $X$, $\kappa \subseteq X \subseteq \beta \kappa$ such that $\CL(X)$ is pseudocompact
	and $X$ is not $(\omega_1 , \omega^*)$-pseudocompact.
\end{example}

\begin{proof}We apply Theorem \ref{NonPseudo} with $\kappa$ and $\theta=\omega_1$, where  $\mathcal A$ be a partition of $\omega_1$ into $\omega_1$ subsets of cardinality $\omega$.
	
	Let $\mathcal C$ be a sequence of pairwise disjoint finite nonempty subsets of $\kappa$. We must show that there exists $E \in [\omega]^\omega$ and $A \in \mathcal A$ such that $|A\cap \bigcup_{n \in E}\mathcal C_n|<\omega.$

	Let $E=\omega$. Since $\bigcup_{n \in \omega} \mathcal C_n$ is countable, there exists $A \in \mathcal A$ such that $A\cap \bigcup_{n \in \omega} \mathcal C_n=\emptyset$.

\end{proof}

Finally, we show that if $\omega\subseteq X\subseteq \beta \omega$ and $\CL(X)$ is pseudocompact, then all powers of $\CL(X)$ are pseudocompact. In fact, we show that $\CL(X)$ is $(\kappa, \omega^*)$-pseudocompact for every $\kappa<\mathfrak h$.
\begin{theorem}
	If $\omega \subseteq X \subseteq \beta \omega$ and $\CL(X)$ is pseudocompact then $\CL(X)$ is $(\kappa, \omega^*)$-pseudocompact for every $\kappa<\mathfrak h$.
\end{theorem}

\begin{proof}Let $\kappa<\mathfrak h$ and $E=[\omega]^{<\omega}\setminus \{0\}$.
	 Since $E$ is a dense subset of $\CL(X)$, it suffices to show that $E^\kappa$ is relatively countably compact in $\CL(X)^\kappa$. So let $A_\alpha$ $(\alpha<\kappa)$ be a collection of sequences on $E$.
	
	 For each $\alpha<\kappa$, let $$\mathcal U_\alpha=\{I \in [\omega]^\omega: \exists m \in \omega \text{ s.t. } A_\alpha \text{ admits a nice dissection over } I\setminus m\}.$$
	
	 By Lemma \ref{secc}(a), each $\mathcal U_\alpha$ is open dense, and since $\kappa<\mathfrak h$, there exists $I \in \bigcap\{\mathcal U_\alpha: \alpha<\kappa\}$. Let $(U_\alpha, D_\alpha)$ be a nice dissection of $A_\alpha$ over $I\setminus m^0_\alpha$ for some $m^0_\alpha \in \omega$.
	
	 For each $\alpha<\kappa$, $$\mathcal V_\alpha=\{J \in [I]^\omega: D_\alpha|(J\setminus m^0_\alpha) \text{ is eventually empty or eventually not empty}\}.$$

	Since each $\mathcal V_\alpha$ is open dense on $[I]^\omega$, there exists $J \in \bigcap\{\mathcal V_\alpha: \alpha<\kappa\}$. So for each $\alpha<\kappa$ there exists $m^1_\alpha\geq m^0_\alpha$ such that either $(\forall n\in J\setminus m^1_\alpha)\, (D_\alpha(n) \neq \emptyset)$ or $(\forall n\in J\setminus m^1_\alpha)\, (D_\alpha(n) = \emptyset)$. Let $T_0=\{\alpha<\kappa: (\forall n\in J\setminus m^1_\alpha)\, (D_\alpha(n) \neq \emptyset)\}$ and $T_1=\{\alpha<\kappa: (\forall n\in J\setminus m^1_\alpha)\,(D_\alpha(n) = \emptyset)\}$.
	
	For every $\alpha \in T_0$ and $n \in J\setminus  m^1_\alpha$, let $f_\alpha(n)=\min D_\alpha(n)$, $g_\alpha(n)=\max D_\alpha(n)$. If $n<m^1_\alpha$, let $f_\alpha(n)=g_\alpha(n)=0$.  Since for each $\alpha \in T_0$
	
	$$\mathcal W_\alpha=\{Z \in [J]^\omega: f_\alpha|X \text{ and } g_\alpha|X \text{ are eventually strictly increasing}\}$$
	
	 is an open dense set, there exists $Z \in [J]^\omega$ and a family of natural numbers $(m_\alpha: \alpha \in T_0)$ such that $m_\alpha\geq m^1_\alpha$, $m_\alpha \in Z\setminus m^1_\alpha$ and $f_\alpha|(Z\setminus m_\alpha), g_\alpha|(Z\setminus m_\alpha)$ are strictly increasing. For $\alpha \in T_1$, let $m_\alpha=m^1_\alpha$. 
	
	Let $j:\omega\rightarrow Z$ be a strictly increasing bijection. For each $\alpha\leq \kappa$, let $\tilde D_\alpha=D_\alpha \circ j$, $\tilde U_\alpha=U_\alpha \circ j$, $\tilde A_\alpha=A_\alpha \circ j$ and $\tilde m_\alpha= j^{-1}(m_\alpha)$. For $\alpha \in T_0$, let $\tilde f_\alpha=f_\alpha\circ j, \tilde g_\alpha=g_\alpha\circ j$. Then $(\tilde D_\alpha, \tilde U_\alpha)$ is a nice dissection of $\tilde A_\alpha$ over $\omega\setminus \tilde m_\alpha$. Also, if $n\in \omega\setminus \tilde m_\alpha$, then if $\alpha \in T_0$, $\tilde D_\alpha(n)\neq \emptyset$, $\tilde f_\alpha(m)=\min \tilde D_\alpha(m)\leq \max \tilde D_\alpha(m)=\tilde g_\alpha(m)$, and if $\alpha \in T_1$, $\tilde D_\alpha(n)=\emptyset$. Finally, notice that $\tilde f_\alpha|(\omega\setminus \tilde m_\alpha), g_\alpha|(\omega\setminus \tilde m_\alpha)$ are strictly increasing. Since $\mathfrak h\leq \mathfrak b$, there exists $a:\omega\rightarrow \omega$ such that $a\geq^*\tilde f_\alpha, \tilde g_\alpha$ for every $\alpha\in T_0$.

	For each $m \in \omega$, let $$I_m=\{\alpha\in T_0: \forall n\geq m \,(a(n)>\tilde f_\alpha(n), \tilde g_\alpha(n)) \text{ and } \tilde m_\alpha \leq m\}.$$
	
	Notice that $I_m$ is a growing sequence of subsets of $T_0$. Recursively, we define a strictly increasing sequence $n_k$ of natural numbers satisfying: $n_0=0$ and, for every $k$, $\forall \alpha \in I_{n_k}\, a(n_k)<\tilde f_\alpha(n_{k+1})\leq \tilde g_\alpha(n_{k+1})<a(n_{k+1}).$
	
	Suppose we have defined $n_0, \dots, n_k$. Let $n_{k+1}=2a(n_k)+1$. Given $\alpha \in I_{n_k}$, $a(k)<\tilde f_\alpha(n_{k+1})$ since $\tilde f_\alpha|(\omega\setminus n_k)$ is strictly increasing and $n_{k+1}>n_k$, it follows that:
	$$a(n_k)<a(n_k)+1\leq f_\alpha(2a(n_k)+1)=f_\alpha(n_{k+1})\leq g_\alpha(n_{k+1})<a(n_{k+1}).$$
	
	Let $C_0$ be the discrete interval $[0, a_0]$, and for each $k$, let  $C_{k+1}=[a_{n_k}+1, a_{n_k+1}]$. Then the $C_k$'s are nonempty, finite and $\omega$ is their disjoint union. Notice that $\tilde D_\alpha(n_k)\subset C_k$ for each $\alpha \in T_0$ and $k \in \omega$. Since $\CL(X)$ is pseudocompact, there exists $p \in \omega^*$ such that $p-\lim C_k$ exists. Then by Lemma \ref{ContZp}, given $\alpha \in T_0$ and $f:\omega\rightarrow \omega $ such that $f(k) \in \tilde D_\alpha(n_k) \subseteq C_k$ for each $k$, it follows that $p-\lim f \in X$, which implies, again by Lemma \ref{ContZp}, that there exists $\exists p-\lim \tilde D_\alpha(n_k) \in \CL(X)$. $\tilde U_\alpha$ converges when $\tilde U_\alpha(n)$ is not eventually empty, therefore $p-\lim \tilde A_\alpha$ exists for every $\alpha<\kappa$, which completes the proof.
\end{proof}

\section{Questions}
Example \ref{exampleH} show that there is a space $X$ such that $\omega\subseteq X\subseteq \beta X$, $X^\kappa$ is countably compact for every $\kappa<\mathfrak h$ and $\CL(X)$ is not pseudocompact. Example \ref{exampleT} shows that if $\mathfrak t<\mathfrak c$ and $2^\mathfrak t=\mathfrak c$ then there exists a space $X$ such that $\omega\subseteq X\subseteq \beta X$, $X^\mathfrak t$ is countably compact and $\CL(X)$ is not pseudocompact. This is true on Cohen's model, where $\mathfrak c=2^\mathfrak t$ and $\mathfrak t=\mathfrak h=\omega_1$ (see \cite{CombinatorialBlass}, \cite{kunen2011set}), however, the following remains open:

\begin{question}Is it consistent that $\mathfrak t<\mathfrak h<\mathfrak c$ and there a space $X$ such that $\omega\subseteq X\subseteq \beta X$, $X^\mathfrak h$ is countably compact and $\CL(X)$ is not pseudocompact?\end{question}

Theorem \ref{EverySmallKappa} states that if $\omega\subseteq X\subseteq \beta \omega$ and $\CL(X)$ is pseudocompact, then $X$ is $(\kappa, \omega^*)$-pseudocompact for every $\kappa<\mathfrak h$. Example \ref{ExistsB} shows that there exists $X$ such that $\omega\subseteq X\subseteq \beta \omega$, $\CL(X)$ is pseudocompact but $X$ is not $(\mathfrak b, \omega^*)$-pseudocompact.

\begin{question}Is there a space $X$ such that $\omega\subseteq X\subseteq \beta \omega$, $\CL(X)$ is pseudocompact but $X$ is not $(\mathfrak h, \omega^*)$-pseudocompact?\end{question}

\section*{Acknowledgements}
All the authors received support from FAPESP (process numbers 2014/16955-2 - post doctoral project, 2015/15166-7 - master's project,  and 2016/26216-8 - research project). The third author also received support from CNPq (307130/2013-4 - produtividade em pesquisa).

\bibliographystyle{amsplain}

\end{document}